\documentclass{amsart}

\usepackage{amssymb,amsmath,amsthm,latexsym,stmaryrd}
\usepackage[pagebackref=true,colorlinks=true]{hyperref}
\usepackage{color}
\numberwithin{equation}{section}
\newtheorem{thm}{Theorem}[section]
\newtheorem{prop}[thm]{Proposition}

\newtheorem{lem}[thm]{Lemma}
\newtheorem{defn}[thm]{Definition}

\newcommand{\thmref}[1]{Theorem~\ref{#1}}
\newcommand{\propref}[1]{Proposition~\ref{#1}}
\newcommand{\lemref}[1]{Lemma~\ref{#1}}

\newcommand{\ie}{i.\hspace{.5pt}e.\ }
\newcommand{\tr}{\operatorname{tr}}

\newcommand{\M}{(\mathcal{M},\allowbreak{}\phi,\allowbreak{}\xi,\allowbreak{}\eta,g)}
\newcommand{\MM}{\mathcal{M}}

\newcommand{\Df}{\dot{D}}
\newcommand{\Tf}{\dot{T}}

\newcommand{\Qf}{\dot{Q}}
\newcommand{\Ds}{\ddot{D}}
\newcommand{\Ts}{\ddot{T}}

\newcommand{\Qs}{\ddot{Q}}

\newcommand{\HH}{\mathcal{H}}
\newcommand{\VV}{\mathcal{V}}

\newcommand{\Span}{{\rm span}}

\newcommand{\ff}{\phi}
\newcommand{\F}{\mathcal{F}}
\newcommand{\UU}{\mathcal{U}}

\newcommand{\n}{\nabla}
\newcommand{\N}{\widehat{N}}

\newcommand{\ta}{\theta}

\newcommand{\lm}{\lambda}

\newcommand{\om}{\omega}
\newcommand{\D}{\mathrm{d}}

\newcommand{\g}{\tilde{g}}
\newcommand{\sx}{\mathop{\mathfrak{S}}}

\newcommand{\s}{\mathop{\mathfrak{S}}}

\begin{document}

\vspace{2cm}

\title[Second natural connection on Riemannian $\Pi$-manifolds]
{Second natural connection on Riemannian $\Pi$-manifolds}

\author{Hristo Manev}
\address{Medical University of Plovdiv, Faculty of Pharmacy,
Department of Medical Physics and Biophysics,   15-A Vasil Aprilov
Blvd.,   Plovdiv 4002,   Bulgaria;}
\email{hristo.manev@mu-plovdiv.bg}



\begin{abstract}
A natural connection, determined by a property of its torsion tensor, is defined and it is called the second natural connection on Riemannian $\Pi$-manifold, \ie the uniqueness of this connection is proved and a necessary and sufficient condition for coincidence with the first natural connection on the considered manifolds is found. The form of the torsion tensor of the second natural connection is obtained in the classes of the Riemannian $\Pi$-manifolds in which it differs from the first natural connection. An explicit example of dimension 5 is given in support of the proven assertions.
\end{abstract}
\maketitle

\section{Introduction}\label{sect-0}
Objects of investigation in the present work are the almost paracontact almost paracomplex Riemannian manifolds, also known as Riemannian $\Pi$-manifolds \cite{ManSta01,ManVes18}. These manifolds are odd-dimensional and they have traceless induced almost product structure on the paracontact distribution. Moreover, the restriction on the paracontact distribution of the almost paracontact structure is an almost paracomplex structure. The beginning of their investigation is given in \cite{ManSta01} by the name almost paracontact Riemannian manifolds of type $(n,n)$, followed by series of papers (e.g. \cite{ManVes18,IvMan2,HMan3,HM17}).

An important role in the geometry of the manifolds with additional tensor structures play the so-called natural connections, \ie affine connections which preserve the structure tensors and the metric (e.g. \cite{KobNom,Ale-Gan2,Gan-Mi,Mek-P-con,Man-Gri2,StaGri}). In \cite{HM21}, we presented and studied the first natural connection $\Df$ on Riemannian $\Pi$-manifolds.
Here, we introduce a natural connection $\Ds$ determined by a property of its torsion tensor $\Ts$ and we call it the second natural connection on the considered manifolds. We prove the uniqueness of $\Ds$ and we determine a necessary and sufficient condition that $\Ds$ coincides with $\Df$. Then, we obtain the form of $\Ts$ in the classes of a known classification of Riemannian $\Pi$-manifolds where $\Ds$ differs from $\Df$.

The paper is organized as follows.
After the present introductory Section 1, Section 2 gives some preliminary facts about Riemannian $\Pi$-manifolds and recall some definitions and assertions for the first natural connection on the studied manifolds necessary for further investigations. Moreover, we characterize all the basic classes of the considered classification of Riemannian $\Pi$-manifolds with respect to $\Tf$. In Section 3, we define $\Ds$ and determine the class of the studied manifolds for which $\Ds$ coincides with $\Df$. After that, we characterize again all the basic classes but now regarding $\Ts$.
The final Section 4 is devoted to an explicit 5-dimensional example in support of the proven theory.

\section{Preliminaries}
\subsection{Riemannian $\Pi$-Manifolds}\label{sect-1}

Let us denote a {Riemannian $\Pi$-manifold} by $\M$, \ie $\mathcal{M}$ is a $(2n+1)$-di\-men\-sional differentiable manifold with a Riemannian $\Pi$-structure $(\ff,\xi,\eta)$ consisting of a (1,1)-tensor field $\ff$, a Reeb vector field $\xi$ and its dual 1-form $\eta$ as well as $\mathcal{M}$ is equipped with a Rie\-mannian metric $g$ such that the following basic identities and their immediately derived properties are valid:
\begin{equation}\label{strM}
\begin{array}{c}
\ff\xi = 0,\qquad \ff^2 = I - \eta \otimes \xi,\qquad
\eta\circ\ff=0,\qquad \eta(\xi)=1,\\ \
\tr \ff=0,\qquad g(\ff x, \ff y) = g(x,y) - \eta(x)\eta(y),
\end{array}
\end{equation}
\begin{equation}\label{strM2}
\begin{array}{ll}
g(\ff x, y) = g(x,\ff y),\qquad &g(x, \xi) = \eta(x),
\\
g(\xi, \xi) = 1,\qquad &\eta(\n_x \xi) = 0,
\end{array}
\end{equation}
where $I$ and $\n$ denote the identity and the Levi-Civita connection of $g$, respectively (\cite{Sato76,ManVes18}). In the latter equalities and further, $x$, $y$, $z$ stand for arbitrary dif\-fer\-en\-tiable vector fields on $\mathcal{M}$ or tangent vectors at a point of $\MM$.

The associated metric $\g$ of $g$ on $\M$ is an indefinite metric of signature $(n + 1, n)$ and compatible with the manifold in the same way as $g$. It is is defined by $\g(x,y)=g(x,\ff y)+\eta(x)\eta(y)$.

On an arbitrary Riemannian $\Pi$-manifold $\M$, we consider two com\-ple\-men\-tary distributions of $T\MM$ using $\xi$ and $\eta$---the horizontal distribution $\HH=\ker(\eta)$ and the vertical distribution $\VV=\Span(\xi)$. They are mutually orthogonal with respect to $g$ and $\tilde{g}$, \ie
\begin{equation}\label{HHVV}
\HH\oplus\VV =T\MM,\qquad \HH\;\bot\;\VV,\qquad \HH\cap\VV=\{o\},
\end{equation}
where $o$ stands for the zero vector field on $\MM$.
Thus the respective horizontal and vertical projectors are determined by $\mathrm{h}:T\MM\mapsto\HH$ and $\mathrm{v}:T\MM\mapsto\VV$.

An arbitrary vector field $x$ is decomposed in the so-called horizontal and vertical component ($x^{\mathrm{h}}$ and $x^{\mathrm{v}}$, respectively) as follows
\begin{equation}\label{hv}
x=x^{\mathrm{h}}+x^{\mathrm{v}},
\end{equation}
where
\begin{equation}\label{Xhv}
x^{\mathrm{h}}=\ff^2x, \qquad x^{\mathrm{v}}=\eta(x)\xi.
\end{equation}

The $(0,3)$-tensor field $F$, defined by
\begin{equation}\label{F}
F(x,y,z)=g\bigl( \left( \nabla_x \ff \right)y,z\bigr),
\end{equation}
plays an important role in the differential geometry of the considered ma\-ni\-folds.
Moreover, the following general properties of $F$ are valid: \cite{ManSta01}
\begin{equation}\label{F-prop}
\begin{array}{l}
F(x,y,z)=F(x,z,y)=-F(x,\ff y,\ff z) +\eta(y)F(x,\xi,z)+\eta(z)F(x,y,\xi),\\[6pt]
F(x,y,\ff z)=-F(x,\ff y, z)+\eta(z)F(x,\ff y,\xi) +\eta(y)F(x,\ff z,\xi),\\[6pt]
F(x,\ff y,\ff z)=-F(x,\ff^2 y,\ff^2 z),\\[6pt]
F(x,\ff y,\ff ^2 z)=-F(x,\ff^2 y,\ff z).
\end{array}
\end{equation}

\begin{lem}[\cite{ManVes18}]\label{lem-F}
The following identities are valid:
\begin{enumerate}
\item[$1)$] $(\nabla_x \eta)(y)=g( \nabla_x \xi,y)$,
\item[$2)$] $\eta(\nabla_x \xi)=0$,
\item[$3)$] $F(x,\ff y,\xi)=-(\nabla_x \eta)(y)$.
\end{enumerate}
\end{lem}

Let $\left(g^{ij}\right)$ be the inverse matrix of $\left(g_{ij}\right)$ of $g$ with respect to
a basis $\left\{\xi;e_i\right\}$ of $T_p\mathcal{M}$ $(i=1,2,\dots,2n; p\in \mathcal{M})$. Then the 1-forms $\theta$, $\theta^*$, $\omega$, called Lee forms, are associated with $F$ and defined by:
\begin{equation}\label{ta-ta*}
\theta=g^{ij}F(e_i,e_j,\cdot),\quad
\theta^*=g^{ij}F(e_i,\ff e_j,\cdot), \quad \omega=F(\xi,\xi,\cdot).
\end{equation}
Taking into account \eqref{F-prop}, the following relations for the Lee forms are valid: \cite{ManSta01}
\begin{equation*}\label{ta-prop}
\begin{array}{l}
\om(\xi)=0,\qquad \ta^*\circ\ff=-\ta\circ\ff^2,\qquad \ta^*\circ\ff^2=\ta\circ\ff.
\end{array}
\end{equation*}

A classification of Riemannian $\Pi$-ma\-ni\-folds with respect to the fundamental tensor $F$, consisting of eleven basic classes $\F_1$, $\F_2$, $\dots$, $\F_{11}$, is given in \cite{ManSta01}. The intersection of the basic classes is the special class $\F_0$ determined by the condition $F=0$. The characteristic conditions of the basic classes in the considered clas\-si\-fi\-cation are: \cite{ManSta01,ManVes18}
%\begin{subequations}\label{Fcon}
\begin{equation}\label{Fcon}
\begin{split}
\F_{1}:\quad &F(x,y,z)=\dfrac{1}{2n}\bigl\{g(\ff x,\ff y)\ta(\ff^2 z)+g(\ff x,\ff z)\ta(\ff^2 y)\\%[6pt]
&\phantom{F(x,y,z)=\dfrac{1}{2n}\bigl\{}-g( x,\ff y)\ta(\ff z)-g( x,\ff z)\ta(\ff y)\bigr\};\\[6pt]
\F_{2}:\quad &F(\xi,y,z)=0,\; F(x,\xi,z)=0,\; \ta=0\\
               &F(x,y,\ff z)+ F(y,z, \ff x)+ F(z, x, \ff y)=0;\\[6pt]
\F_{3}:\quad &F(\xi,y,z)=0,\; F(x,\xi,z)=0,\\
              & F(x,y,z)+F(y,z,x)+F(z,x,y)=0;\\[6pt]
\F_{4}:\quad &F(x,y,z)=\dfrac{\ta(\xi)}{2n}\bigl\{g(\ff x,\ff y)\eta(z)+g(\ff x,\ff z)\eta(y)\bigr\};\\[6pt]
\F_{5}:\quad &F(x,y,z)=\dfrac{\ta^*(\xi)}{2n}\bigl\{g( x,\ff y)\eta(z)+g(x,\ff z)\eta(y)\bigr\};\\[6pt]
\F_{6}:\quad &F(x,y,z)=F(x,y,\xi)\eta(z)+F(x,z,\xi)\eta(y),\\[6pt]
                &F(x,y,\xi)=F(y,x,\xi)=F(\ff x,\ff y,\xi); \\[6pt]
\F_{7}:\quad &F(x,y,z)=F(x,y,\xi)\eta(z)+F(x,z,\xi)\eta(y),\\[6pt]
                &F(x,y,\xi)=-F(y,x,\xi)=F(\ff x,\ff y,\xi); \\[6pt]
\F_{8}:\quad &F(x,y,z)=F(x,y,\xi)\eta(z)+F(x,z,\xi)\eta(y),\\[6pt]
                &F(x,y,\xi)=F(y,x,\xi)=-F(\ff x,\ff y,\xi); \\[6pt]
\F_{9}:\quad &F(x,y,z)=F(x,y,\xi)\eta(z)+F(x,z,\xi)\eta(y),\\[6pt]
                &F(x,y,\xi)=-F(y,x,\xi)=-F(\ff x,\ff y,\xi); \\[6pt]
\F_{10}:\quad &F(x,y,z)=-\eta(x)F(\xi,\ff y,\ff z); \\[6pt]
\F_{11}:\quad
&F(x,y,z)=\eta(x)\left\{\eta(y)\om(z)+\eta(z)\om(y)\right\}.
\end{split}
\end{equation}
%\end{subequations}

In \cite{ManVes18}, the $(1,2)$-tensors $N$ and $\N$ determined by
\begin{equation*}\label{N-nff}
\begin{split}
N(x,y)=&\left(\n_{\ff x}\ff\right)y-\ff\left(\n_{x}\ff\right)y-\left(\n_{x}\eta\right)(y)\xi\\[6pt]
& -\left(\n_{\ff
y}\ff\right)x+\ff\left(\n_{y}\ff\right)x+\left(\n_{y}\eta\right)(x)\xi,
\end{split}
\end{equation*}
\begin{equation*}\label{N1-nff}
\begin{split}
\N(x,y)=&\left(\n_{\ff x}\ff\right)y-\ff\left(\n_{x}\ff\right)y-\left(\n_{x}\eta\right)(y)\xi\\[6pt]
& +\left(\n_{\ff
y}\ff\right)x-\ff\left(\n_{y}\ff\right)x-\left(\n_{y}\eta\right)(x)\xi
\end{split}
\end{equation*}
are called Nijenhuis tensor and associated Nijenhuis tensor, respectively, for the $\Pi$-structure on $\MM$. The tensors $N$ and $\N$ are antisymmetric and symmetric, respectively, \ie the following properties hold
\begin{equation*}\label{NN-prop}
\begin{array}{l}
N(x,y) = -N(y,x),\qquad
\N(x, y) = \N(y, x).
\end{array}
\end{equation*}
The corresponding $(0,3)$-tensors of $N$ and $\N$ on $\M$ are defined by
\begin{equation*}
\begin{array}{ll}
N(x,y,z)=g\left(N(x,y),z\right), \qquad
\N(x,y,z)=g\left(\N(x,y),z\right)
\end{array}
\end{equation*}
and have the following properties: \cite{ManVes18}
%\begin{subequations}\label{N-prop}
\begin{equation}\label{N-prop}
\begin{array}{l}
N(\ff^2 x, \ff y,\ff z)= - N(\ff^2 x, \ff^2 y,\ff^2 z),  \\[6pt]
 N(\ff^2 x, \ff^2 y,\ff^2 z) =
N(\ff x, \ff y,\ff^2 z), \\[6pt]
N( x, \ff^2 y,\ff^2 z)= - N(x, \ff y,\ff z), \\[6pt]
 N(\ff^2 x, \ff^2 y, z) = N(\ff x, \ff y,
z), \\[6pt]
%\end{array}
%\end{equation}
%\begin{equation}
%\begin{array}{l}
N(\xi, \ff y,\ff z) = {-} N(\xi, \ff^2 y,\ff^2 z), \\[6pt]
 N(\ff x, \ff y, \xi) = N(\ff^2 x,
\ff^2 y, \xi),
\end{array}
\end{equation}
%\end{subequations}
\begin{equation}\label{N1-prop}
\begin{array}{l}
\N(\ff^2 x, \ff y,\ff z)= - \N(\ff^2 x, \ff^2 y,\ff^2 z), \\[6pt] \N(\ff^2 x, \ff^2
y,\ff^2 z) = \N(\ff x, \ff y,\ff^2 z), \\[6pt]
\N( x, \ff^2 y,\ff^2 z)= - \N(x, \ff y,\ff z), \\[6pt] \N(\ff^2 x, \ff^2 y, z) =
\N(\ff x, \ff y, z),\\[6pt]
\N(\xi, \ff y,\ff z) = - \N(\xi, \ff^2 y,\ff^2 z), \\[6pt] \N(\ff x, \ff y, \xi) =
\N(\ff^2 x, \ff^2 y, \xi).
\end{array}
\end{equation}

The tensors $N$ and $\N$ are expressed by means of $F$ through the equalities: \cite{ManVes18}
\begin{equation}\label{NN1-F}
\begin{array}{ll}
N(x,y,z)=F(\ff x,y,z)-F(\ff y,x,z)-F(x,y,\ff z)+F(y,x,\ff z)\\[6pt]
\phantom{N(x,y,z)=}+\eta(z)\left\{F(x,\ff y,\xi)-F(y,\ff
x,\xi)\right\},\\[6pt]
\N(x,y,z)=F(\ff x,y,z)+F(\ff y,x,z)-F(x,y,\ff z)-F(y,x,\ff z)\\[6pt]
\phantom{\N(x,y,z)=}+\eta(z)\left\{F(x,\ff y,\xi)+F(y,\ff x,\xi)\right\}.
\end{array}
\end{equation}

Vice versa, $F$ is expressed in terms of $N$ and $\N$ as follows: \cite{ManVes18}
\begin{equation}\label{F=NN}
\begin{array}{l}
F(x,y,z)=\dfrac14\bigl\{N(\ff x,y,z)+N(\ff x,z,y)+\N(\ff x,y,z)+\N(\ff x,z,y)\bigr\}\\[6pt]
\phantom{F(x,y,z)=}-\dfrac12\eta(x)\bigl\{N(\xi,y,\ff z)+\N(\xi,y,\ff z)+\eta(z)\N(\xi,\xi,\ff y)\bigr\}.
\end{array}
\end{equation}

Let us remark that the class of the normal Riemannian $\Pi$-manifolds, \ie with the condition $N=0$, is $\UU_0=\F_1\oplus\F_2 \oplus \F_4\oplus\F_5\oplus\F_6$. On the other hand, the class with the property $\N=0$ is $\widehat{\UU }_0=\F_3\oplus\F_7$. Applying the expression of $F$ from \eqref{F=NN} for these two classes, we obtain:
\begin{equation*}\label{F-U0}
\begin{array}{ll}
\UU_0: \quad & F(x,y,z)=\dfrac14\bigl\{\N(\ff x,y,z)+\N(\ff x,z,y)\bigr\}\\[6pt]
& \phantom{F(x,y,z)=}-\dfrac12\eta(x)\bigl\{\N(\xi,y,\ff z)+\eta(z)\N(\xi,\xi,\ff y)\bigr\},\\[6pt]
\widehat{\UU}_0: \quad & F(x,y,z)=\dfrac14\bigl\{N(\ff x,y,z)+N(\ff x,z,y)\bigr\}-\dfrac12\eta(x)N(\xi,y,\ff z).
\end{array}
\end{equation*}

The class $\widehat{\UU}_0$ is important for further considerations, as well as its orthogonal complement one, characterized by the following
\begin{lem}\label{Nfifi=0}
The class $\UU_1=\F_1\oplus\F_2\oplus\F_4\oplus\F_5\oplus\F_6\oplus\F_8\oplus\F_9\oplus\F_{10}\oplus\F_{11}$
of the Riemannian $\Pi$-manifolds $\M$ is determined by the condition
\[
N(\ff\cdot,\ff\cdot)=0.
\]
\end{lem}
\begin{proof}
Using the expression of $N$ in terms of $F$ given in \eqref{NN1-F} and the characteristic conditions \eqref{Fcon} of the basic classes of the considered manifolds, we establish the truthfulness of the lemma.
\end{proof}

It is obvious from \lemref{Nfifi=0} that the condition $N( \ff\cdot,\ff\cdot) \neq 0$ is valid for $\widehat{\UU}_0$. So, we can conclude that the following two properties of $N$ and $\N$ are equivalent:
\[
N(\ff\cdot,\ff\cdot) \neq 0 \qquad \Leftrightarrow \qquad \N(\cdot,\cdot) = 0.
\]

Let us denote by $T$ the {torsion tensor} of an arbitrary affine connection $D$, \ie
\begin{equation}\label{T-def}
T(x,y)=D_xy-D_yx-[x,y].
\end{equation}

The corresponding $(0,3)$-tensor with respect to the metric $g$ is determined by
\begin{equation}\label{D-T-03}
\begin{array}{l}
T(x,y,z)=g(T(x,y),z).
\end{array}
\end{equation}

The associated 1-forms of $T$, denoted by $t$, $t^*$ and $\hat{t}$, are defined by
\begin{equation}\label{t}
\begin{array}{l}
t(x)=g^{ij}T(x,e_i,e_j),\qquad t^*(x)=g^{ij}T(x,e_i,\ff e_j),\qquad \hat{t}(x)=T(x,\xi,\xi)
\end{array}
\end{equation}
with respect to a basis $\left\{\xi;e_i\right\}$ of $T_p\mathcal{M}$ $(i=1,2,\dots,2n; p\in \mathcal{M})$.

\subsection{First natural connection on Riemannian $\Pi$-Manifolds}
In \cite{HM21}, we defined a non-symmetric natural connection and called it the first natural connection on a Riemannian $\Pi$-manifold. We obtained relations between the introduced connection and the Levi-Civita connection, as well as we study some of its curvature characteristics in the so-called main classes, \ie these basic classes in which $F$ is expressed explicitly by the metrics and the Lee forms.
Firstly, we recall some definitions and assertions from \cite{HM21} necessary for further investigations.

Let us consider an arbitrary Riemannian $\Pi$-manifold $\M$.

\begin{defn}[\cite{HM21}]
An affine connection $D$ on a Riemannian $\Pi$-manifold $\M$ is called a {natural connection} for the Riemannian $\Pi$-structure $(\ff,\xi,\eta,g)$ if this structure is parallel with respect to $D$, \ie $D\ff=D\xi=D\eta=Dg=0$.
\end{defn}
As a consequence, the associated metric $\tilde{g}$ is also parallel with respect to $D$ on $\M$, \ie $D\tilde{g}=0$.

Let $Q$ stand for the potential of $D$ with respect to $\n$:
\begin{equation}\label{1}
D_xy=\n_xy+Q(x,y),
\end{equation}

\begin{equation}\label{2.2}
Q(x,y,z)=g\left(Q(x,y),z\right).
\end{equation}

\begin{prop}[\cite{HM21}]\label{thm-Q}
An affine connection $D$ is a natural connection on a Riemannian $\Pi$-manifold if and only if the following properties hold:
\begin{equation*}
\begin{array}{c}
Q(x,y,\ff z)-Q(x,\ff y,z)=F(x,y,z),\\[6pt]
Q(x,y,z)=-Q(x,z,y).
\end{array}
\end{equation*}
\end{prop}

\begin{defn}[\cite{HM21}]\label{defn-D1}
A natural connection $\Df$, defined by
\begin{equation}\label{D1}
\begin{array}{l}
\Df_xy=\n_xy-\dfrac{1}{2}\bigl\{\left(\n_x\ff\right)\ff
y-\left(\n_x\eta\right)y\cdot\xi\bigr\}-\eta(y)\n_x\xi,
\end{array}
\end{equation}
is called the {first natural connection} on a Riemannian $\Pi$-manifold $\M$.
\end{defn}

As we remarked in \cite{HM21}, the restriction of $\Df$ on the paracontact distribution $\HH$ of $\M$ is the known $P$-connection on the corresponding Riemannian manifold equipped with traceless almost product structure (see e.g. \cite{Mek-P-con}).

Let $\Tf$ denote the torsion tensor of $\Df$, \ie
$$
\Tf(x,y)=\Df_xy-\Df_yx-[x,y].
$$
Then, according to \cite{HM21}, we have
\begin{equation}\label{D1-T}
\begin{array}{l}
\Tf(x,y)=-\dfrac{1}{2}\left\{(\n_x\ff)\ff y-(\n_y\ff)\ff x -\D \eta (x,y)\xi\right\}+ \eta(x)\n_y \xi - \eta(y)\n_x \xi.
\end{array}
\end{equation}

The corresponding $(0,3)$-tensor with respect to $g$ is determined as follows \cite{HM21}
\begin{equation}\label{D1-T-03}
\begin{array}{l}
\Tf(x,y,z)=g(\Tf(x,y),z)
\end{array}
\end{equation}
and it is expressed by $F$ through
\begin{equation}\label{D1-Txyz}
\begin{array}{l}
\Tf(x,y,z)=-\dfrac{1}{2}\left\{F(x,\ff y,z)-F(y,\ff x,z)\right\}\\[6pt]
\phantom{\Tf(x,y,z)=}-\dfrac{1}{2}\eta(z)\left\{F(x,\ff y,\xi)-F(y,\ff x,\xi)\right\}\\[6pt]
\phantom{\Tf(x,y,z)=}+\eta(y)F(x,\ff z,\xi)-\eta(x)F(y,\ff z,\xi).
\end{array}
\end{equation}

Moreover, in \cite{HM21}, the torsion of $\Df$ with respect to $N$ and $\N$ is obtained:
\begin{subequations}\label{T=NN}
\begin{equation}%\label{T=NN}
\begin{array}{l}
\Tf(x,y,z)=-\dfrac{1}{8}\bigl\{2N(\ff x,\ff y,z)+N(\ff x,z,\ff y)-N(\ff y,z,\ff x)\\[6pt]
\phantom{\Tf(x,y,z)=-\dfrac{1}{8}\bigl\{}
+\N(\ff x,z,\ff y)-\N(\ff y,z,\ff x)\bigr\}\\[6pt]
\phantom{\Tf(x,y,z)=}
+\dfrac{1}{4}\eta(x)\bigl\{2N(\xi,\ff y,\ff z)-N(\ff y,\ff z,\xi)\\[6pt]
\phantom{\Tf(x,y,z)=+\dfrac{1}{4}\eta(x)\bigl\{}
+2\eta(z)\N(\xi,\xi,\ff^2 y)-\N(\ff y,\ff z,\xi)\bigr\}\\[6pt]
\phantom{\Tf(x,y,z)=}
-\dfrac{1}{4}\eta(y)\bigl\{2N(\xi,\ff x,\ff z)-N(\ff x,\ff z,\xi)\\[6pt]
\phantom{\Tf(x,y,z)=-\dfrac{1}{4}\eta(y)\bigl\{}
+2\eta(z)\N(\xi,\xi,\ff^2 x)-\N(\ff x,\ff z,\xi)\bigr\}\\[6pt]
\end{array}
\end{equation}
\begin{equation}
\begin{array}{l}
\phantom{\Tf(x,y,z)=}
-\dfrac{1}{8}\eta(z)\bigl\{2N(\ff x,\ff y,\xi)+N(\ff x,\xi,\ff y)-N(\ff y,\xi,\ff x)\\[6pt]
\phantom{\Tf(x,y,z)=-\dfrac{1}{4}\eta(y)\bigl\{}
+\N(\ff x,\xi,\ff y)-\N(\ff y,\xi,\ff x)\bigr\}.
\end{array}
\end{equation}
\end{subequations}
The latter result, using the decomposition in \eqref{HHVV}, \eqref{hv} and \eqref{Xhv}, is transformed into the following form with respect to the horizontal and the vertical components of the vector fields:
%\begin{subequations}
\begin{equation}\label{T1Nhv}
\begin{split}
\Tf(x,y,z) =-\dfrac{1}{8}\bigl\{
&\sx N(x^{\mathrm{h}},y^{\mathrm{h}},z^{\mathrm{h}}) + N(x^{\mathrm{h}},y^{\mathrm{h}},z^{\mathrm{h}}) \\[6pt]
&+ \N(y^{\mathrm{h}},z^{\mathrm{h}},x^{\mathrm{h}})-\N(z^{\mathrm{h}},x^{\mathrm{h}},y^{\mathrm{h}})\bigr\} \\[6pt]
-\dfrac{1}{4}\bigl\{
&2N(x^{\mathrm{h}},y^{\mathrm{h}},z^{\mathrm{v}})+N(y^{\mathrm{h}},z^{\mathrm{v}},x^{\mathrm{h}})
+N(z^{\mathrm{v}},x^{\mathrm{h}},y^{\mathrm{h}}) \\[6pt]
&+2N(x^{\mathrm{v}},y^{\mathrm{h}},z^{\mathrm{h}}) + N(y^{\mathrm{h}},z^{\mathrm{h}},x^{\mathrm{v}})
+2N(x^{\mathrm{h}},y^{\mathrm{v}},z^{\mathrm{h}})  \\[6pt]
&+ N(z^{\mathrm{h}},x^{\mathrm{h}},y^{\mathrm{v}})+2\N(y^{\mathrm{h}},z^{\mathrm{h}},x^{\mathrm{v}})
-\N(z^{\mathrm{v}},x^{\mathrm{h}},y^{\mathrm{h}}) \\[6pt]
&-\N(z^{\mathrm{h}},x^{\mathrm{h}},y^{\mathrm{v}})-2\N(z^{\mathrm{v}},x^{\mathrm{v}},y^{\mathrm{h}})
+ 2\N(y^{\mathrm{v}},z^{\mathrm{v}},x^{\mathrm{h}})\bigr\},
\end{split}
\end{equation}
%\end{subequations}
where $\sx$ stands for the cyclic sum by the three arguments.

Taking into account that for the basic classes in $\widehat{\UU}_0$ the tensor $N$ has the form \begin{equation}\label{N-1-11}
\begin{array}{ll}
\F_3: \quad & N(x,y,z)=-2\bigl\{F(\ff x,\ff y,\ff z)+F(\ff^2 x,\ff^2 y,\ff z)\bigr\},\\[6pt]
\F_7: \quad & N(x,y,z)=4F(x,\ff y,\xi)\eta(z)
\end{array}
\end{equation}
and $\N$ vanishes \cite{ManVes18}, as a consequence of \eqref{T1Nhv}, we obtain
%\begin{subequations}\label{T1Nhv-37}
\begin{equation}\label{T1Nhv-37}
\begin{split}
\widehat{\UU}_0:\quad \Tf(x,y,z) =-\dfrac{1}{8}\bigl\{
&\sx N(x^{\mathrm{h}},y^{\mathrm{h}},z^{\mathrm{h}}) + N(x^{\mathrm{h}},y^{\mathrm{h}},z^{\mathrm{h}})\bigr\} \\[6pt]
\phantom{\Tf(x,y,z) =}-\dfrac{1}{4}\bigl\{
&\sx N(x^{\mathrm{h}},y^{\mathrm{h}},z^{\mathrm{v}}) +N(x^{\mathrm{h}},y^{\mathrm{h}},z^{\mathrm{v}})\bigr\}
\end{split}
\end{equation}
%\end{subequations}
and respectively
\begin{equation}\label{T1N-F3F7}
\begin{split}
\F_3:\quad \Tf(x,y,z)=-\dfrac{1}{8}\bigl\{
&\sx N(x^{\mathrm{h}},y^{\mathrm{h}},z^{\mathrm{h}}) +N(x^{\mathrm{h}},y^{\mathrm{h}},z^{\mathrm{h}})\bigr\},\\[6pt]
\F_7:\quad \Tf(x,y,z)=-\dfrac{1}{4}\bigl\{
&\sx N(x^{\mathrm{h}},y^{\mathrm{h}},z^{\mathrm{v}}) +N(x^{\mathrm{h}},y^{\mathrm{h}},z^{\mathrm{v}})\bigr\}.
\end{split}
\end{equation}
Then, by virtue of \lemref{lem-F}, \eqref{Fcon}, \eqref{N-prop}, \eqref{N-1-11} and denoting $N(x^{\mathrm{h} },y^{\mathrm{h}},z^{\mathrm{h}})=N^{\mathrm{h}}(x,y,z)$, equalities \eqref{T1Nhv-37} and \eqref{T1N-F3F7} take the following shorter form:
\begin{equation*}\label{T1N-F3F7-abbr}
\begin{split}
&\widehat{\UU}_0:\quad \Tf=-\dfrac{1}{8}\bigl\{\s N^{\mathrm{h}}+N^{\mathrm{h}}\bigr\}+\dfrac{1}{2}\bigl\{\eta\wedge\D\eta+\D\eta\otimes\eta\bigr\},\qquad\\[6pt]
\end{split}
\end{equation*}
\begin{equation*}
\begin{split}
&\F_3:\quad \Tf=-\dfrac{1}{8}\bigl\{\s N^{\mathrm{h}}+N^{\mathrm{h}}\bigr\},\qquad\\[6pt]
&\F_7:\quad \Tf=\dfrac{1}{2}\bigl\{\eta\wedge\D\eta+\D\eta\otimes\eta\bigr\}.
\end{split}
\end{equation*}

Similarly to \eqref{t}, the torsion forms $\dot{t}$, $\dot{t}^*$ and $\widehat{\dot{t}}$ for $\Tf$ with respect to a basis $\left\{\xi;e_i\right\}$ of $T_p\mathcal{M}$ $(i=1,2,\dots,2n; p\in \mathcal{M})$ are defined by: \cite{HM21}
\begin{equation}\label{t1}
\begin{array}{l}
\dot{t}(x)=g^{ij}\Tf(x,e_i,e_j),\qquad \dot{t}^*(x)=g^{ij}\Tf(x,e_i,\ff e_j),\qquad
\widehat{\dot{t}}(x)=\Tf(x,\xi,\xi).
\end{array}
\end{equation}

The following formulae for the torsion forms hold
\begin{equation}\label{t-fB}
\begin{array}{c}
\dot{t}(x)=\dfrac{1}{2}\ta(\ff x)-\ta^*(\xi)\eta(x),\\[6pt]
\dot{t}^*(x)=\dfrac{1}{2}\ta^*(\ff x)-\ta(\xi)\eta(x),\\[6pt]
\widehat{\dot{t}}(x)=\om(\ff x)
\end{array}
\end{equation}
and the following relations between them and the Lee forms are valid \cite{HM21}
%\begin{subequations}\label{t1t1*}
\begin{equation}\label{t1t1*}
\begin{array}{c}
\dot{t}^*\circ\ff=\dot{t}\circ\ff^2,\\[6pt]
2\dot{t}\circ\ff=\ta\circ\ff^2,\qquad
2\dot{t}\circ\ff^2=\ta\circ\ff,\\[6pt]
2\dot{t}^*\circ\ff=\ta^*\circ\ff^2,\qquad
2\dot{t}^*\circ\ff^2=\ta^*\circ\ff.
\end{array}
\end{equation}
%\end{subequations}

\begin{thm}\label{thm:FiT1}
Let $\M$ be a $(2n+1)$-dimensional Riemannian $\Pi$-manifold. Then the basic classes $\F_i$ $(i=1,\dots,11)$ are characterised by the following properties of the torsion tensor $\Tf$ of the first natural connection $\Df$:
\[
\begin{array}{rl}
\F_1:\; &\Tf(x,y)=-\dfrac{1}{2n}\left\{\dot{t}(\ff^2 y)\ff^2 x-\dot{t}(\ff^2 x)\ff^2 y
+\dot{t}(\ff x)\ff y-\dot{t}(\ff y)\ff x\right\}; \\[6pt]
\F_2:\; &\Tf(\xi,y)=0,\quad \eta\left(\Tf(x,y)\right)=0, \quad \Tf(x,y)=-\Tf(\ff x,\ff y),\quad \dot{t}=0;\\[6pt]
\F_3:\; &\Tf(\xi,y)=0,\quad \eta\left(\Tf(x,y)\right)=0,\quad \Tf(x,y)=-\ff \Tf(x,\ff y);\\[6pt]
\F_4:\; &\Tf(x,y)=-\dfrac{1}{2n}\dot{t}^*(\xi)\left\{\eta(y)\ff x-\eta(x)\ff y\right\};\\[6pt]
\F_5:\; &\Tf(x,y)=-\dfrac{1}{2n}\dot{t}(\xi)\left\{\eta(y)\ff^2 x-\eta(x)\ff^2 y\right\};\\[6pt]
\F_6:\; &\Tf(x,y)=\eta(x)\Tf(\xi,y)-\eta(y)\Tf(\xi,x)-\eta(\Tf(x,y))\xi,\\[6pt]
&\Tf(\xi,y,z)=-\Tf(\xi,z,y)=\Tf(\xi,\ff y,\ff z)=\dfrac{1}{2}\Tf(y,z,\xi)=\dfrac{1}{2}\Tf(\ff y,\ff z,\xi);\\[6pt]
\F_{7}:\; &\Tf(x,y)=\eta(x)\Tf(\xi,y)-\eta(y)\Tf(\xi,x)+\eta(\Tf(x,y))\xi,\\[6pt]
            &\Tf(\xi,y,z)=-\Tf(\xi,z,y)=\Tf(\xi,\ff y,\ff z)=\dfrac{1}{2}\Tf(y,z,\xi)=-\dfrac{1}{2}\Tf(\ff y,\ff z,\xi);\\[6pt]
\F_{8}:\; &\Tf(x,y)=\eta(x)\Tf(\xi,y)-\eta(y)\Tf(\xi,x)-\eta(\Tf(x,y))\xi,\\[6pt]
            &\Tf(\xi,y,z)=-\Tf(\xi,z,y)=-\Tf(\xi,\ff y,\ff z)=\dfrac{1}{2}\Tf(y,z,\xi)=-\dfrac{1}{2}\Tf(\ff y,\ff z,\xi);\\[6pt]
%
%\end{array}
%\]
%\[
%\begin{array}{rl}
\F_{9}:\; &\Tf(x,y)=\eta(x)\Tf(\xi,y)-\eta(y)\Tf(\xi,x),\\[6pt]
            &\Tf(\xi,y,z)=\Tf(\xi,z,y)=-\Tf(\xi,\ff y,\ff z);\\[6pt]
\F_{10}:\; &\Tf(x,y)=\eta(x)\Tf(\xi,y)-\eta(y)\Tf(\xi,x),\\[6pt]
            &\Tf(\xi,y,z)=-\Tf(\xi,z,y)=-\Tf(\xi,\ff y,\ff z);\\[6pt]
\F_{11}:\; &\Tf(x,y)=\left\{\eta(y)\widehat{\dot{t}}(x) - \eta(x)\widehat{\dot{t}}(y)\right\}\xi.%\\[6pt]
\end{array}
\]
\end{thm}
\begin{proof}
The truthfulness of the assertions in the cases when $\M$ belongs to the main classes $\F_1$, $\F_4$, $\F_5$, $\F_{11}$ is proved in \cite{HM21}.

Now, let $\M \in \F_2$. Taking into account the characteristic conditions \eqref{Fcon} of $F$ in the considered class, the expression of $\Tf$ from \eqref{D1-Txyz} takes the following form:
\begin{equation}\label{T1-F2}
\Tf(x,y,z)=-\dfrac{1}{2}\bigl\{F(x,\ff y,z)-F(y,\ff x,z)\bigr\}.
\end{equation}
Therefore, bearing in mind \eqref{D1-T-03} and \eqref{strM2}, we find that the following identities hold:
\[
\Tf(\xi,y)=0,\qquad \eta\left(\Tf(x,y)\right)=0.
\]
Setting $x=\ff x$ and $y=\ff y$ in \eqref{T1-F2} and using \eqref{F-prop} and \eqref{Fcon}, we have
\begin{equation*}\label{T1ff-F2}
\Tf(\ff x,\ff y,z)=\dfrac{1}{2}\bigl\{F(x,\ff y,z)-F(y,\ff x,z)\bigr\}.
\end{equation*}
So, comparing the latter equality and \eqref{T1-F2}, the following property immediately follows
\[
\Tf(x,y)=-\Tf(\ff x,\ff y).
\]
Using the characteristic conditions of $\F_2$ from \eqref{Fcon} in \eqref{t-fB}, we get $\dot{t}=0$, with which we have established the properties of $\F_2$ with respect to $\Tf$.
The other cases are proved in a similar way.
\end{proof}

\section{Second natural connection on Riemannian $\Pi$-Manifolds}

Let $\Ts$ denote the torsion tensor of a natural connection $\Ds$, \ie according to \eqref{T-def} and \eqref{D-T-03}, we have
\begin{equation}\label{T2-defn}
\begin{array}{l}
\Ts(x,y)=\Ds_xy-\Ds_yx-[x,y],\\[6pt]
\Ts(x,y,z)=g(\Ts(x,y),z).
\end{array}
\end{equation}
Also, let the following property hold for $\Ts$:
%\begin{subequations}
\begin{equation}\label{T2-prop}
\begin{split}
&\Ts(x,y,z)+\Ts(y,z,x)+\Ts(\ff x, y,\ff z)+\Ts(y,\ff z,\ff x)\\[6pt]
    &-\eta(x)\left\{\Ts(\xi,y,z)+\Ts(y,z,\xi)-\eta(y)\Ts(\xi,z,\xi)\right\}\\[6pt]
    &-\eta(y)\left\{\Ts(x,\xi,z)+\Ts(\xi,z,x)+\Ts(\ff x,\xi,\ff z)+\Ts(\xi,\ff z,\ff x)\right\}\\[6pt]
    &-\eta(z)\left\{\Ts(x,y,\xi)+\Ts(y,\xi,x)-\eta(y)\Ts(x,\xi,\xi)\right\}=0.
\end{split}
\end{equation}
%\end{subequations}

\begin{defn}\label{defn-D2}
A natural connection $\Ds$ for which \eqref{T2-prop} holds is called the second natural connection on a Riemannian $\Pi$-manifold $\M$.
\end{defn}

Let us remark that the restriction of $\Ds$ on the paracontact distribution $\HH$ of $\M$ is another studied natural connection (called canonical connection) on the corresponding Riemannian manifold equipped with traceless almost product structure (see e.g. \cite{Mihova}).

\begin{thm}\label{thm-1vtora}
On an arbitrary Riemannian $\Pi$-manifold $\M$ there exists a unique second natural connection.
\end{thm}
\begin{proof}
Let us construct an affine connection $\Ds$ on a Riemannian $\Pi$-ma\-ni\-fold $\M$ as follows:
\begin{equation}\label{D=nQ}
g(\Ds_xy,z)=g(\n_xy,z)+\Qs(x,y,z),
\end{equation}
where the potential $\Qs$ of $\Ds$ is determined by
\begin{equation}\label{Q-can-F}
\begin{split}
\Qs(x,y,z)=&\;\Qf(x,y,z)-\dfrac{1}{8}\left\{N(\ff^2 z,\ff^2 y,\ff^2 x)+2\eta(x)N(\ff z,\ff y,\xi)\right\}.
\end{split}
\end{equation}
%a $N$ е тензорът на Нийенхойс, определен чрез \eqref{NF}.

Using \eqref{Q-can-F}, we verify that $\Ds$ satisfies the conditions of \propref{thm-Q}, \ie $\Ds$ is a natural connection on $\M$.

From \eqref{T2-defn}, the symmetry of $\n$ and the analogous definitions of \eqref{1} and \eqref{2.2} for $\Ds$, it follows that
\begin{equation*}\label{T2QQ}
\Ts(x,y,z)=\Qs(x,y,z)-\Qs(y,x,z).
\end{equation*}

Taking into account the latter equality, \eqref{Q-can-F} and \eqref{T=NN}, we obtain
\begin{equation}\label{T2=T1+}
\begin{split}
\Ts(x,y,z)=\Tf(x,y,z)&-\dfrac{1}{8}\left\{N(\ff^2 z,\ff^2 y,\ff^2 x)-N(\ff^2 z,\ff^2
x,\ff^2 y)\right.\\[6pt]
&\phantom{+\dfrac{1}{8}\left\{\right.}\left.+2\eta(x)N(\ff z,\ff
y,\xi)-2\eta(y)N(\ff z,\ff x,\xi)\right\},
\end{split}
\end{equation}
which is equivalent to
\begin{equation}\label{T2-Nhv}
\begin{split}
\Ts(x,y,z)=\ \Tf(x,y,z)&+\dfrac{1}{8}\bigl\{\sx N(x^{\mathrm{h}},y^{\mathrm{h}},z^{\mathrm{h}}) - N(x^{\mathrm{h}},y^{\mathrm{h}},z^{\mathrm{h}})\bigr\}\\[6pt]
&+\dfrac{1}{4}\bigl\{\sx N(x^{\mathrm{h}},y^{\mathrm{h}},z^{\mathrm{v}}) - N(x^{\mathrm{h}},y^{\mathrm{h}},z^{\mathrm{v}})\bigr\}.
\end{split}
\end{equation}

Substituting the form of $\Tf$ from \eqref{D1-Txyz} into \eqref{T2=T1+}, we get
%\begin{subequations}\label{T2-F}
\begin{equation}\label{T2-F}
\begin{split}
\Ts(x,y,z)=&-\dfrac{1}{2}\left\{F(x,\ff y,z)-F(y,\ff x,z)\right\}\\[6pt]
&-\dfrac{1}{2}\eta(z)\left\{F(x,\ff y,\xi)-F(y,\ff x,\xi)\right\}\\[6pt]
&+\eta(y)F(x,\ff z,\xi)-\eta(x)F(y,\ff z,\xi)\\[6pt]
&+\dfrac{1}{8}\left\{N(\ff^2 z,\ff^2 y,\ff^2 x)-N(\ff^2 z,\ff^2
x,\ff^2 y)\right.\\[6pt]
&\phantom{+\dfrac{1}{8}\left\{\right.}\left.+2\eta(x)N(\ff z,\ff y,\xi)-2\eta(y)N(\ff z,\ff x,\xi)\right\}.
\end{split}
\end{equation}
%\end{subequations}

Then, bearing in mind \eqref{F-prop} and \eqref{N-prop}, we obtain by direct verification that \eqref{T2-prop} holds for $\Ds$. Therefore, \eqref{D=nQ} and \eqref{Q-can-F} determine the natural connection $\Ds$ which is the second natural connection on $\M$.

Taking into account \eqref{D=nQ}, \eqref{Q-can-F} and the corresponding form of $\Qf$ and $N$ from \eqref{D1}
and \eqref{NN1-F}, we have an explicit expression of $\Ds$ in terms of $F$. This fact uniquely defines the considered manifold and proves the uniqueness of the second natural connection.
\end{proof}

Using \eqref{NN1-F}, the formula in \eqref{T2-F} takes the following form:
%\begin{subequations}\label{T2-F+F}
\begin{equation}\label{T2-F+F}
\begin{split}
\Ts(x,y,z)=&-\dfrac{1}{2}\left\{F(x,\ff y,z)-F(y,\ff x,z)\right\}\\[6pt]
&-\dfrac{1}{2}\eta(z)\left\{F(x,\ff y,\xi)-F(y,\ff x,\xi)\right\}\\[6pt]
&+\eta(y)F(x,\ff z,\xi)-\eta(x)F(y,\ff z,\xi)\\[6pt]
&-\dfrac{1}{8}\left\{2F(\ff^2 z,\ff^2 x,\ff y)+F(\ff x,\ff^2 z,\ff^2 y)-F(\ff y,\ff^2 z,\ff^2 x)\right.\\[6pt]
&\phantom{-\dfrac{1}{8}\left\{\right.}\left.-F(\ff^2 x,\ff^2 z,\ff y)+F(\ff^2 y,\ff^2 z,\ff x)\right\}\\[6pt]
&-\dfrac{1}{4}\eta(x)\left\{F(\ff^2 z,\ff y,\xi)-F(\ff^2 y,\ff z,\xi)\right.\\[6pt]
&\phantom{-\dfrac{1}{4}\eta(x)\left\{\right.}\left.+F(\ff z,\ff^2 y,\xi)-F(\ff y,\ff^2 z,\xi)\right\}\\[6pt]
&+\dfrac{1}{4}\eta(y)\left\{F(\ff^2 z,\ff x,\xi)-F(\ff^2 x,\ff z,\xi)\right.\\[6pt]
&\phantom{+\dfrac{1}{4}\eta(y)\left\{\right.}\left.+F(\ff z,\ff^2 x,\xi)-F(\ff x,\ff^2 z,\xi)\right\}.
\end{split}
\end{equation}
%\end{subequations}

Similarly to \eqref{t}, we define the torsion forms $\ddot{t}$, $\ddot{t}^*$ and $\widehat{\ddot{t}}$ for $\Ts$ with respect to a basis $\left\{\xi;e_i\right\}$ of $T_p\mathcal{M}$ $(i=1,2,\dots,2n; p\in \mathcal{M})$:
\begin{equation}\label{t2}
\begin{array}{l}
\ddot{t}(x)=g^{ij}\Ts(x,e_i,e_j),\qquad
\ddot{t}^*(x)=g^{ij}\Ts(x,e_i,\ff e_j),\qquad
\widehat{\ddot{t}}(x)=\Ts(x,\xi,\xi).
\end{array}
\end{equation}

By dint of \eqref{T2=T1+}, \eqref{T2-F+F}, \eqref{t2}, \eqref{t-fB} and $\eta(e_i)=0$ $(i=1,\dots,2n)$, we obtain the following expression of $\ddot{t}$ in terms of $\ta$, $\ta^*$ and $\om$ of $\M$:
\begin{equation}\label{t-2}
\ddot{t}(x)=\dot{t}(x)=\dfrac{1}{2}\ta(\ff x)-\ta^*(\xi)\eta(x).
\end{equation}
By an analogous approach, we calculate the form of $\ddot{t}^*$ and $\widehat{\ddot{t}}$:
\begin{equation}\label{t*-om-2}
\begin{array}{l}
\ddot{t}^*(x)=\dot{t}^*(x)=\dfrac{1}{2}\ta^*(\ff x)-\ta(\xi)\eta(x),\qquad \widehat{\ddot{t}}(x)=\widehat{\dot{t}}(x)=\om(\ff x).
\end{array}
\end{equation}

Thus, we obtained that the torsion forms of the first and second natural connection coincide.

Similarly to \eqref{t1t1*}, we obtain the following relation between $\ddot{t}$ and $\ddot{t}^*$:
\begin{equation*}\label{t2t2*}
\ddot{t}^*\circ\ff=\ddot{t}\circ\ff^2.
\end{equation*}

\begin{thm}\label{thm-D1=D2}
The first natural connection $\Df$ coincides with the second natural connection $\Ds$ if and only if
$\M \in \UU_1$.
\end{thm}
\begin{proof}
From \eqref{Q-can-F} we immediately establish that a necessary and sufficient condition for $\Df$ and $\Ds$ to coincide is $N(\ff\cdot,\ff\cdot)=0$.
According to \lemref{Nfifi=0}, the latter equality defines the class $\UU_1$ of the Riemannian $\Pi$-manifolds. Thus, we proved the truthfulness of the theorem.
\end{proof}

Let us remark that the assertions proved in \cite{HM21} for $\Df$ are also valid for $\Ds$ when $\M$ belongs to $\UU_1$, bearing in mind \thmref{thm-D1=D2}.

Considering \eqref{T1Nhv}, we obtain the following formula for the torsions of $\Df$ and $\Ds$ on $\M \in \UU_1$:
\[
\begin{split}
\Ts(x,y,z)=\Tf(x,y,z)=&-\dfrac{1}{8}\bigl\{
\N(y^{\mathrm{h}},z^{\mathrm{h}},x^{\mathrm{h}})-\N(z^{\mathrm{h}},x^{\mathrm{h}},y^{\mathrm{h}})\bigr\} \\[6pt]
&-\dfrac{1}{4}\bigl\{
N(y^{\mathrm{h}},z^{\mathrm{v}},x^{\mathrm{h}})+N(z^{\mathrm{v}},x^{\mathrm{h}},y^{\mathrm{h}}) \\[6pt]
&\phantom{-\dfrac{1}{4}\bigl\{}-\N(z^{\mathrm{v}},x^{\mathrm{h}},y^{\mathrm{h}}) - \N(z^{\mathrm{h}},x^{\mathrm{h}},y^{\mathrm{v}})\bigr\}\\[6pt]
&-\dfrac{1}{2}\bigl\{
N(x^{\mathrm{v}},y^{\mathrm{h}},z^{\mathrm{h}}) +N(x^{\mathrm{h}},y^{\mathrm{v}},z^{\mathrm{h}}) \\[6pt]
&\phantom{-\dfrac{1}{4}\bigl\{}+\N(y^{\mathrm{h}},z^{\mathrm{h}},x^{\mathrm{v}})-\N(z^{\mathrm{v}},x^{\mathrm{v}},y^{\mathrm{h}})  + \N(y^{\mathrm{v}},z^{\mathrm{v}},x^{\mathrm{h}})\bigr\}.
\end{split}
\]

The torsion tensors $\Tf$ and $\Ts$ differ from each other when the Riemannian $\Pi$-manifold belongs to the classes $\F_3$ and $\F_7$ or to their direct sums with other classes.
Using \eqref{T1Nhv-37} and \eqref{T2-Nhv}, we obtain the form of $\Ts$ when $\M$ is from the class $\widehat{\UU}_0$, as well as when the considered manifold is from the separate two basic classes $\F_3$ and $\F_7$:
\begin{equation*}\label{T2N-F3F7}
\begin{split}
\widehat{\UU}_0:\quad \Ts(x,y,z)=&-\dfrac{1}{4} N(x^{\mathrm{h}},y^{\mathrm{h}},z^{\mathrm{h}})-\dfrac{1}{2} N(x^{\mathrm{h}},y^{\mathrm{h}},z^{\mathrm{v}}),\\[6pt]
\F_3:\quad \Ts(x,y,z)=&-\dfrac{1}{4}N(x^{\mathrm{h}},y^{\mathrm{h}},z^{\mathrm{h}}),\\[6pt]
\F_7:\quad \Ts(x,y,z)=&-\dfrac{1}{2}N(x^{\mathrm{h}},y^{\mathrm{h}},z^{\mathrm{v}}).
\end{split}
\end{equation*}
The latter three formulae, using \lemref{lem-F}, \eqref{Fcon}, \eqref{N-prop} and \eqref{N-1-11}, take the following shorter form:
\begin{equation*}\label{T2N-F3F7-abbr}
\begin{split}
&\widehat{\UU}_0:\quad \Ts=-\dfrac{1}{4} N^{\mathrm{h}}+\D\eta\otimes\eta,\\[6pt]
&\F_3:\quad \Ts=-\dfrac{1}{4} N^{\mathrm{h}},\\[6pt]
&\F_7:\quad \Ts=\D\eta\otimes\eta.
\end{split}
\end{equation*}

\begin{thm}\label{thm:FiT2}
Let $\M$ be a $(2n+1)$-dimensional Riemannian $\Pi$-manifold. Then the basic classes $\F_i$ $(i=1,\dots,11)$ are characterised by the following properties of the torsion tensor $\Ts$ of the second natural connection $\Ds$:
\[
\begin{array}{rl}
\F_1:\; &\Ts(x,y)=-\dfrac{1}{2n}\left\{\ddot{t}(\ff^2 y)\ff^2 x-\ddot{t}(\ff^2 x)\ff^2 y +\ddot{t}(\ff x)\ff y-\ddot{t}(\ff y)\ff x\right\}; \\[6pt]
\F_2:\; &\Ts(\xi,y)=0,\quad \eta\left(\Ts(x,y)\right)=0,\quad \Ts(x,y)=-\Ts(\ff x,\ff y),\quad \ddot{t}=0;\\[6pt]
\F_3:\; &\Ts(\xi,y)=0,\quad \eta\left(\Ts(x,y)\right)=0,\quad \Ts(x,y)=-\ff \Ts(x,\ff y);\\[6pt]
\F_4:\; &\Ts(x,y)=-\dfrac{1}{2n}\ddot{t}^*(\xi)\left\{\eta(y)\ff x-\eta(x)\ff y\right\};\\[6pt]
\end{array}
\]
\[
\begin{array}{rl}
\F_5:\; &\Ts(x,y)=-\dfrac{1}{2n}\ddot{t}(\xi)\left\{\eta(y)\ff^2 x-\eta(x)\ff^2 y\right\};\\[6pt]
\F_6:\; &\Ts(x,y)=\eta(x)\Ts(\xi,y)-\eta(y)\Ts(\xi,x)-\eta\left(\Ts(x,y)\right)\xi,\\[6pt]
&\Ts(\xi,y,z)=-\Ts(\xi,z,y)=\Ts(\xi,\ff y,\ff z)=\dfrac{1}{2}\Ts(y,z,\xi)=\dfrac{1}{2}\Ts(\ff y,\ff z,\xi);\\[6pt]
\F_{7}:\; &\Ts(x,y)=\eta(x)\Ts(\xi,y)-\eta(y)\Ts(\xi,x)+\eta\left(\Ts(x,y)\right)\xi,\\[6pt]
            &\Ts(\xi,y,z)=-\Ts(\xi,z,y)=\Ts(\xi,\ff y,\ff z),\quad \Ts(y,z,\xi)=\Ts(\ff y,\ff z,\xi);\\[6pt]
\F_{8}:\; &\Ts(x,y)=\eta(x)\Ts(\xi,y)-\eta(y)\Ts(\xi,x)-\eta\left(\Ts(x,y)\right)\xi,\\[6pt]
            &\Ts(\xi,y,z)=-\Ts(\xi,z,y)=-\Ts(\xi,\ff y,\ff z)=\dfrac{1}{2}\Ts(y,z,\xi)=-\dfrac{1}{2}\Ts(\ff y,\ff z,\xi);\\[6pt]
\F_{9}:\; &\Ts(x,y)=\eta(x)\Ts(\xi,y)-\eta(y)\Ts(\xi,x),\\[6pt]
            &\Ts(\xi,y,z)=\Ts(\xi,z,y)=-\Ts(\xi,\ff y,\ff z);\\[6pt]
\F_{10}:\; &\Ts(x,y)=\eta(x)\Ts(\xi,y)-\eta(y)\Ts(\xi,x),\\[6pt]
            &\Ts(\xi,y,z)=-\Ts(\xi,z,y)=-\Ts(\xi,\ff y,\ff z);\\[6pt]
\F_{11}:\; &\Ts(x,y)=\left\{\eta(y)\widehat{\ddot{t}}(x) - \eta(x)\widehat{\ddot{t}}(y)\right\}\xi.%\\[6pt]
\end{array}
\]
\end{thm}
\begin{proof}
In the cases when $\M \in \UU_1$, $\Df$ coincides with $\Ds$, according to \thmref{thm-D1=D2}. Then the conditions for $\Ts$ coincide with the corresponding conditions for $\Tf$ from \thmref{thm:FiT1}.

Let us now suppose that $\M \in \F_3$. Taking into account \eqref{Fcon} in the considered class, the expression of $\Ts$ from \eqref{T2-F+F} takes the following form:
%\begin{subequations}\label{T2-F3}
\begin{equation}\label{T2-F3}
\begin{split}
\Ts(x,y,z)=&-\dfrac{1}{2}\left\{F(x,\ff y,z)-F(y,\ff x,z)\right\}\\[6pt]
%\end{split}
%\end{equation}
%\begin{equation}
%\begin{split}
&+\dfrac{1}{8}\left\{F(\ff y,\ff z,\ff x)+F(\ff^2 y,z,\ff x)\right.\\[6pt]
&\phantom{+\dfrac{1}{8}\left\{\right.}\left.-F(\ff x,\ff z,\ff y)-F(\ff^2 x,z,\ff y)\right\}.
\end{split}
\end{equation}
%\end{subequations}
Therefore, using \eqref{T2-defn} and \eqref{strM2}, we find that the following identities hold:
\[
\Ts(\xi,y)=0,\qquad \eta\left(\Ts(x,y)\right)=0.
\]
By virtue of \eqref{T2-F3}, \eqref{F-prop} and \eqref{Fcon}, we obtain
\[
\begin{split}
\Ts(x,\ff y,\ff z)=\,&\dfrac{1}{2}\left\{F(x,\ff y,z)-F(y,\ff x,z)\right\}\\[6pt]
%\end{split}
%\]
%\[
%\begin{split}
&-\dfrac{1}{8}\left\{F(\ff y,\ff z,\ff x)+F(\ff^2 y,z,\ff x)\right.\\[6pt]
&\phantom{-\dfrac{1}{8}\left\{\right.}\left.-F(\ff x,\ff z,\ff y)-F(\ff^2 x,z,\ff y)\right\}.
\end{split}
\]
So, comparing the latter equality and \eqref{T2-F3}, it immediately follows
\[
\Ts(x,y)=-\Ts(\ff x,\ff y),
\]
with which we established the characteristics of $\F_3$ with respect to $\Ts$.

Similarly, we get the conditions for $\Ts$ in the case when $\M \in \F_7$.
\end{proof}

\section{Example}\label{sect-4}
In \cite{HM21}, a 5-dimensional Lie group $\mathcal{G}$ equipped with an invariant Riemannian $\Pi$-structure $(\phi, \xi, \eta, g)$ is considered. The constructed Riemannian $\Pi$-manifold belongs to the basic class $\F_4$. Therefore, $(\mathcal{G}, \phi, \xi, \eta, g)$ is a manifold from the class $\UU_1$, i.e., by virtue of \thmref{thm-D1=D2}, the first natural connection $\Df$ and the second natural connection $\Ds$ coincide. Therefore, all the assertions made for $\Df$ in this example also hold for $\Ds$.

Let us consider a Lie group $\mathcal{L}$ of dimension $5$ which has a basis $\{e_0,\dots, e_{4}\}$ of left-invariant vector fields on $\mathcal{L}$ and let the corresponding Lie algebra be defined by the following commutators
\begin{equation*}\label{comL}
\begin{array}{l}
[e_1,e_2] = [e_3,e_4] = \lm_1 e_1 + \lm_2 e_2 + \lm_3 e_3+ \lm_4 e_4 + 2\mu_1 e_0,\\[6pt]
[e_1,e_4] = - [e_2,e_3] = \lm_3 e_1 + \lm_4 e_2 + \lm_1 e_3+ \lm_2 e_4 + 2\mu_2 e_0,\\[6pt]
\end{array}
\end{equation*}
where $\lm_i, \mu_j\in\mathbb{R}$ $(i=1,2,3,4; j=1,2)$ and $[e_k,e_l]=0$ in the other cases.

Let $(\ff,\xi,\eta)$ be an invariant $\Pi$-structure defined by
\begin{equation}\label{strEx2}
\begin{array}{l}
\xi=e_0, \quad \ff  e_1=e_{3},\quad  \ff e_2=e_{4},\quad \ff  e_3=e_{1},\quad \ff  e_4=e_{2}, \\[6pt]
\eta(e_1)=\eta(e_2)=\eta(e_3)=\eta(e_4)=0,\quad \eta(e_0)=1.
\end{array}
\end{equation}
Let $g$ stand for a Riemannian metric determined by
\begin{equation}\label{gEx2}
\begin{array}{l}
g(e_i,e_i)=1, \quad g(e_i,e_j)=0,\quad i,j\in\{0,1,\dots,4\},\; i\neq j.
\end{array}
\end{equation}
Thus, by virtue of \eqref{strM} and \eqref{strM2}, the constructed manifold $(\mathcal{L},\ff, \xi, \eta, g)$ is a Riemannian $\Pi$-manifold.

\begin{thm}\label{thm-ex2}
The Riemannian $\Pi$-manifold $(\mathcal{L},\ff, \xi, \eta, g)$ belongs to the class $\F_7$.
\end{thm}
\begin{proof}
By dint of \eqref{strM}, \eqref{strM2} and the well-known Koszul equality regarding $g$ and $\n$, we obtain the components of $\n$ as follows:
\begin{subequations}\label{n-ex2}
\begin{equation}%\label{n-ex2}
\begin{array}{l}
\n_{e_0} e_0 = 0,\\[6pt]
\n_{e_0} e_1 = \n_{e_1} e_0 = -\mu_1 e_2-\mu_2 e_4,\\[6pt]
\n_{e_0} e_2 = \n_{e_2} e_0 = \mu_1 e_1+\mu_2 e_3,\\[6pt]
\n_{e_0} e_3 = \n_{e_3} e_0 = -\mu_2 e_2-\mu_1 e_4,\\[6pt]
\n_{e_0} e_4 = \n_{e_4} e_0 = \mu_2 e_1+\mu_1 e_3,\\[6pt]
\n_{e_1} e_1 = \n_{e_3} e_3 = -\lm_1 e_2-\lm_3 e_4,\\[6pt]
\n_{e_1} e_3 = \n_{e_3} e_1 = -\lm_3 e_2-\lm_1 e_4,\\[6pt]
\n_{e_2} e_2 = \n_{e_4} e_4 = \lm_2 e_1+\lm_4 e_3,\\[6pt]
\n_{e_2} e_4 = \n_{e_4} e_2 = \lm_4 e_1+\lm_2 e_3,\\[6pt]
\n_{e_1} e_2 = \n_{e_3} e_4 = \lm_1 e_1+\lm_3 e_3+\mu_1 e_0,\\[6pt]
\end{array}
\end{equation}
\begin{equation}
\begin{array}{l}
\n_{e_1} e_4 = \n_{e_3} e_2 = \lm_3 e_1+\lm_1 e_3+\mu_2 e_0,\\[6pt]
\n_{e_2} e_1 = \n_{e_4} e_3 = -\lm_2 e_2-\lm_4 e_4-\mu_1 e_0,\\[6pt]
\n_{e_2} e_3 = \n_{e_4} e_1 = -\lm_4 e_2-\lm_2 e_4-\mu_2 e_0.
\end{array}
\end{equation}
\end{subequations}

Using \eqref{F}, \eqref{strEx2}, \eqref{gEx2} and \eqref{n-ex2}, we calculate the components $F_{ijk}=F(e_i,e_j,\allowbreak{}e_k)$. Also, we obtain the components of $\ta$, $\ta^*$ and $\om$ by \eqref{ta-ta*}. The nonzero ones of them are determined by the following equalities and their well-known symmetries from \eqref{F-prop}:
\begin{equation*}\label{F-ex2}
\begin{array}{l}
F_{104}=F_{302}=-F_{203}=-F_{401}=\mu_1,\\[6pt]
F_{102}=F_{304}=-F_{201}=-F_{403}=\mu_2.
\end{array}
\end{equation*}

By virtue of the latter equalities, \eqref{Fcon} and the following form of the components of $F$ in $\F_7$, given in \cite{ManVes18},
\begin{equation*}
\begin{split}
&F_{7}(x,y,z)=\dfrac{1}{4}\bigl\{ %
[F(\ff^2 x,\ff^2 y,\xi)-F(\ff^2 y,\ff^2 x,\xi)\\[6pt]
&\phantom{F_{7}(x,y,z)=\dfrac{1}{4}\bigl\{ %
}
+F(\ff x,\ff y,\xi)-F(\ff y,\ff x,\xi)]\eta(z)\\[6pt]
&\phantom{F_{7}(x,y,z)=\dfrac{1}{4}\bigl\{} %
+ [F(\ff^2 x,\ff^2 z,\xi)-F(\ff^2 z,\ff^2 x,\xi)\\[6pt]
&\phantom{F_{7}(x,y,z)=\dfrac{1}{4}\bigl\{[}
+F(\ff x,\ff z,\xi)-F(\ff z,\ff x,\xi)]\eta(y)\bigr\},
\end{split}
\end{equation*}
we establish the truthfulness of the theorem.
\end{proof}

Let us remark that, according to \thmref{thm-ex2}, $(\mathcal{L}, \phi, \xi, \eta, g)$ is a manifold from $\widehat{\UU}_0$, \ie bearing in mind \thmref{thm-D1=D2}, $\Df$ and $\Ds$ on the constructed manifold do not coincide.

Let us consider $\Df$ on $(\mathcal{L},\allowbreak{}\ff, \xi, \eta, g)$ defined by \eqref{D1}. Then, using \eqref{D1}, \eqref{strEx2} and \eqref{n-ex2}, we obtain the components of $\Df$. The nonzero ones are the following:
%\begin{subequations}%\label{Df-ex2}
\[
\begin{array}{ll}
\Df_{e_0}e_1=-\mu_1 e_2-\mu_2 e_4, \quad &
\Df_{e_0}e_2=\mu_1 e_1+\mu_2 e_3, \\[6pt]
\Df_{e_0}e_3=-\mu_2 e_2-\mu_1 e_4, \quad &
\Df_{e_0}e_4=\mu_2 e_1+\mu_1 e_3, \\[6pt]
\Df_{e_1}e_1=\Df_{e_3}e_3=-\lm_1 e_2-\lm_3 e_4, \quad &
\Df_{e_1}e_2=\Df_{e_3}e_4=\lm_1 e_1+\lm_3 e_3, \\[6pt]
\Df_{e_1}e_3=\Df_{e_3}e_1=-\lm_3 e_2-\lm_1 e_4, \quad &
\Df_{e_1}e_4=\Df_{e_3}e_2=\lm_3 e_1+\lm_1 e_3, \\[6pt]
\Df_{e_2}e_1=\Df_{e_4}e_3=-\lm_2 e_2-\lm_4 e_4, \quad &
\Df_{e_2}e_2=\Df_{e_4}e_4=\lm_2 e_1+\lm_4 e_3, \\[6pt]
\Df_{e_2}e_3=\Df_{e_4}e_1=-\lm_4 e_2-\lm_2 e_4, \quad &
\Df_{e_2}e_4=\Df_{e_4}e_2=\lm_4 e_1+\lm_2 e_3.
\end{array}
\]
%\end{subequations}

Using \eqref{strEx2}, \eqref{gEx2}, \eqref{n-ex2}, \eqref{D1-T} and \eqref{D1-T-03}, we calculate the components $\Tf_{ijk}=\Tf(e_i,e_j,e_k)$ of $\Tf$. The nonzero ones of them are determined by the following equalities and their well-known antisymmetries by the first and second argument:
\begin{equation}\label{T1-ex2}
\begin{array}{l}
\Tf_{102}=\Tf_{021}=\Tf_{304}=\Tf_{043}=\dfrac{1}{2}\Tf_{210}=\dfrac{1}{2}\Tf_{430}=\mu_1,\\[6pt]
\Tf_{104}=\Tf_{023}=\Tf_{302}=\Tf_{041}=\dfrac{1}{2}\Tf_{410}=\dfrac{1}{2}\Tf_{230}=\mu_2.
\end{array}
\end{equation}

Taking into account \eqref{t1} and \eqref{T1-ex2}, we obtain $\dot{t}$, $\dot{t}^*$ and $\widehat{\dot{t}}$ and get that they are all zero, \ie
\[
\dot{t}=\dot{t}^*=\widehat{\dot{t}}=0.
\]

Let us now consider $\Ds$ on $(\mathcal{L},\ff, \xi, \eta, g)$ defined by \eqref{T2-defn} and \eqref{T2-prop}. By virtue of \eqref{T1-ex2} and \eqref{T2=T1+}, we obtain the components $\Ts_{ijk}=\Ts(e_i,e_j,e_k)$ of $\Ts$. The nonzero ones of them are determined by the following equalities and their well-known antisymmetries by the first and second argument:
\begin{equation*}\label{T2-ex2}
\begin{array}{l}
\Ts_{210}=\Ts_{430}=2\mu_1,\qquad
\Ts_{410}=\Ts_{230}=2\mu_2.
\end{array}
\end{equation*}

As it is proved in \eqref{t-2} and \eqref{t*-om-2}, the torsion forms of the first and second natural connection coincide, \ie $\ddot{t}$, $\ddot{t}^*$ and $\widehat{\ddot{t}}$ also vanish on $(\mathcal{L},\ff, \xi, \eta, g)$ and the following equalities hold:
\[
\ddot{t}=\ddot{t}^*=\widehat{\ddot{t}}=0.
\]

The obtained results regarding the torsion properties of the constructed Rie\-man\-nian $\Pi$-manifold $(\mathcal{L},\ff, \xi, \eta, g)$ confirm the statements made in \thmref{thm:FiT1} and \thmref{thm:FiT2} in the case of the class $\F_7$.

%%%%%%%%%%%%%%%%%%%%%%%%%%%%%%%%%%%%%%%%%%


\vspace{6pt}

\begin{thebibliography}{999}
\bibitem{ManSta01}
\textsc{M.\,Manev, M.\,Staikova.} \textit{On almost paracontact Rie\-man\-nian manifolds of type $(n,n)$}, Journal of Geometry {\bf 72} (2001), 108--114.

\bibitem{ManVes18}
\textsc{M.\,Manev,  V.\,Tavkova.}
\emph{On the almost paracontact almost pa\-ra\-comp\-lex Rie\-man\-nian manifolds},
Facta Universitatis, Series: Mathematics and Informatics \textbf{33}   (2018), 637--657.

\bibitem{IvMan2}
\textsc{S.\,Ivanov, H. \,Manev, M. \,Manev.}
{\em Para-Sasaki-like Rie\-man\-nian manifolds and new Einstein metrics},
Revista de la Real Academia de Ciencias Exactas, Fisicas y Naturales. Serie A. Matematicas RACSAM {\bf 115} (2021), art.no. 112.

\bibitem{HMan3}
\textsc{H.\,Manev, M.\,Manev.}
\emph{Pair of associated Schouten-van Kam\-pen con\-nec\-tions adapted to an almost paracontact almost paracomplex Riemannian structure}.
Mathematics  {\bf  9} (7) (2021), art. no. 736.

\bibitem{HM17}
\textsc{H.\,Manev, M.\,Manev.}
\emph{Para-Ricci-like solitons on Rie\-man\-nian manifolds with almost paracontact structure and almost paracomplex structure}.
Mathematics  {\bf  9} (14) (2021), art. no. 1704.

\bibitem{KobNom}
\textsc{S.\,Kobayashi, K.\,Nomizu.} \textit{Foundations of differential ge\-om\-e\-try}, vol. I, Wiley-Interscience (1963).


\bibitem{Ale-Gan2}
\textsc{V. Alexiev, G. Ganchev.} \emph{Canonical connection on a con\-for\-mal
almost con\-tact metric manifolds}, Annual of Sofia Uni\-ver\-si\-ty ''St. Kliment Ohridski''. Faculty of Mathematics and Informatics {\bf 81} (1987), no. 1, 29--38.


\bibitem{Gan-Mi}
\textsc{G. Ganchev, V. Mihova.} \emph{Canonical connection and the
ca\-no\-ni\-cal con\-for\-mal group on an almost complex manifold with
$B$-met\-ric}, Annual of Sofia Uni\-ver\-si\-ty ''St. Kliment Ohridski''. Faculty of Mathematics and Informatics {\bf 81}
(1987), 195--206.


 %\bibitem{Gri1}
% \textsc{D.\,Gribacheva.}
% \textit{Natural connections on Riemannian pro\-duct manifolds},
% International Journal of Geometric Methods in Modern Physics {\bf 09} (2012), art. no. 07.







\bibitem{Mek-P-con}
\textsc{D.\,Mekerov.}
\textit{P-connection on Riemannian almost product
ma\-ni\-folds},  Comptes rendus de l’Academie bulgare des Sci\-en\-ces {\bf 62} (2009), 1363--1370.





\bibitem{Man-Gri2}
\textsc{M.\,Manev, K.\,Gribachev.} \textit{Conformally invariant tensors on
almost con\-tact manifolds with $B$-metric}, Serdica Ma\-the\-ma\-ti\-cal Journal {\bf 20} (1994), 133--147.





\bibitem{StaGri}
\textsc{M.\,Staikova, K.\,Gribachev.}
\textit{Canonical connections and their conformal in\-vari\-ants on Riemannian P-manifolds}, Ser\-di\-ca Mathematical Journal {\bf 18} (1992), 150--161.


\bibitem{Sato76}
\textsc{I.\,Sat\={o}.} \textit{On a structure similar to the almost contact struc\-ture},
Tensor (N.S.) {\bf 30}  (1976), 219--224.

\bibitem{HM21}
\textsc{H.\,Manev.}
\emph{First natural connection on Riemannian $\Pi$-manifolds}.
arXiv: 2301.11694.

\bibitem{Mihova}
 \textsc{V.\,Mihova.} \textit{Canonical connections and the canonical con\-for\-mal group on a Riemannian
almost product manifold},
 Serdica Mathematical Journal {\bf 15}  (1989), 351--358.

\end{thebibliography}
\end{document}